\documentclass[12pt]{amsart}
\usepackage{amsmath,amsfonts,amssymb,amsthm}
\usepackage{graphicx,color}
\DeclareMathRadical{\sqrtsign}{symbols}{"70}{largesymbols}{"70}
\newlength{\figboxwidth}             
\newcommand{\grad}{\nabla}

\def\@ifundefined#1#2#3%
  {\expandafter\ifx\csname#1\endcsname\relax#2\else#3\fi}

\@ifundefined{theoremstyle}{
}{
\theoremstyle{plain} 
}
\newtheorem{theorem}{Theorem}[section]

\newtheorem{proposition}[theorem]{Proposition}
\newtheorem{lemma}[theorem]{Lemma}

\@ifundefined{theoremstyle}{
}{
\theoremstyle{definition} 
}
\newtheorem{definition}[theorem]{Definition}

\newtheorem{remark}[theorem]{Remark}

\mathchardef\GG="321D
\newcommand{\mcc}[1]{{}}

\numberwithin{equation}{section}

\title[Hausdorff dimension across generic dynamical spectra II]
{Continuity of Hausdorff dimension across generic dynamical Lagrange and Markov spectra II}

\author[A. Cerqueira]{A. Cerqueira}
\thanks{A.C. was partially supported by CNPq-Brazil. Also, she thanks the hospitality of Coll\`ege de France and IMPA-Brazil during the preparation of this article.}
\address{Aline Cerqueira: 
IMPA, Instituto de Matem\'atica Pura e Aplicada,   
Estrada Dona Castorina, Jardim Bot\^anico, 
Rio de Janeiro, RJ, CEP 22460-320, Brazil
}
\email{alineagc@gmail.com}

\author[C. G. Moreira]{C. G. Moreira}
\thanks{C.G.M. was partially supported by CNPq-Brazil.}
\address{Carlos G. Moreira: 
IMPA, Instituto de Matem\'atica Pura e Aplicada,   
Estrada Dona Castorina, Jardim Bot\^anico, 
Rio de Janeiro, RJ, CEP 22460-320, Brazil
}
\email{gugu@impa.br}

\author[S. Roma\~na]{S. Roma\~na}
\address{Sergio Roma\~na: UFRJ, 
Universidade Federal do Rio de Janeiro, Av. Athos da Silveira Ramos 149, Centro de Tecnologia (Bloco C),  Cidade Universit\'aria, Ilha do Fund\~ao, Rio de Janeiro, RJ, CEP 21941-909, Brazil
}
\email{sergiori@im.ufrj.br}


\date{\today}

\begin{document}

\begin{abstract}
Let $g_0$ be a smooth pinched negatively curved Riemannian metric on a complete surface $N$, and let $\Lambda_0$ be a basic hyperbolic set of the geodesic flow of $g_0$ with Hausdorff dimension strictly smaller than two. Given a small smooth perturbation $g$ of $g_0$ and a smooth real-valued function $f$ on the unit tangent bundle to $N$ with respect to $g$, let $L_{g,\Lambda,f}$, resp. $M_{g,\Lambda,f}$ be the Lagrange, resp. Markov spectrum of  asymptotic highest, resp. highest values of $f$ along the geodesics in the hyperbolic continuation $\Lambda$  of $\Lambda_0$.  

We prove that, for generic choices of $g$ and $f$, the Hausdorff dimension of the sets $L_{g,\Lambda, f}\cap (-\infty, t)$ vary continuously with $t\in\mathbb{R}$ and, moreover, $M_{g,\Lambda, f}\cap (-\infty, t)$ has the same Hausdorff dimension of $L_{g,\Lambda, f}\cap (-\infty, t)$ for all $t\in\mathbb{R}$. 
\end{abstract}

\maketitle

\vspace{-3.1pt}
\setcounter{tocdepth}{1}


\section{Introduction}\label{s.introduction}

The first paper of this series \cite{CMM} discussed the continuity properties of the Hausdorff dimension across dynamical Lagrange and Markov spectra of surface diffeomorphisms. In this article, our goal is to extend the results in \cite{CMM} to the case of geodesic flows of negatively curved Riemannian surfaces. 

\subsection{Dynamical Markov and Lagrange spectra}\label{ss.dynamical-Markov-Lagrange}

Let $M$ be a smooth manifold, $T=\mathbb{Z}$ or $\mathbb{R}$, and $\phi=(\phi^t)_{t\in T}$ be a discrete-time ($T=\mathbb{Z}$) or continuous-time ($T=\mathbb{R}$) smooth dynamical system on $M$, that is, $\phi^t:M\to M$ are smooth diffeomorphisms, $\phi^0=\textrm{id}$, and $\phi^t\circ\phi^s=\phi^{t+s}$ for all $t,s\in T$.   

Given a compact invariant subset $\Lambda\subset M$ and a function $f:M\to\mathbb{R}$, we define the \emph{dynamical Markov, resp. Lagrange, spectrum} $M_{\phi, \Lambda, f}$, resp. $L_{\phi, \Lambda, f}$ as 
$$M_{\phi, \Lambda, f}=\{m_{\phi, f}(x): x\in\Lambda\}, \quad \textrm{resp.} \quad L_{\phi, \Lambda, f}=\{\ell_{\phi, f}(x): x\in\Lambda\}$$
where  
$$m_{\phi, f}(x):=\sup\limits_{t\in T} f(\phi^t(x)), \quad \textrm{resp.} \quad \ell_{\phi, f}(x):=\limsup\limits_{t\to+\infty} f(\phi^t(x))$$

\begin{remark}\label{r.L<M} An elementary compactness argument (cf. Remark in Section 3 of \cite{MoRo}) shows that  
$$\{\ell_{\phi, f}(x):x\in A\}\subset\{m_{\phi, f}(x): x\in A\}\subset f(A)$$
whenever $A\subset M$ is a compact $\phi$-invariant subset. 
\end{remark}

\subsection{Statement of the main result}

In this paper, we shall study the fractal geometry of $M_{\phi, \Lambda, f}\cap (-\infty, t)$ and $L_{\phi, \Lambda, f}\cap (-\infty, t)$ as $t\in\mathbb{R}$ varies in the context of geodesic flows on negatively curved Riemannian surfaces.  

More precisely, let $N$ be a complete surface, let $g_0$ be a smooth ($C^r$, $r\geq 4$) pinched negatively curved Riemannian metric on  $N$, \emph{i.e}, the curvature is bounded above and  below by two negative constants. Let $\phi_{g_0}=(\phi_{g_0}^t)_{t\in\mathbb{R}}$ be the geodesic flow on the unit tangent bundle $M=S_{g_0}N$ of $N$ with respect to $g_0$. Consider a horseshoe $\Lambda_0$ of $\phi_{g_0}$ with Hausdorff dimension $\textrm{dim}(\Lambda_0)<2$. Denote by $\mathcal{U}$ a small ($C^r$, $r\geq 4$) neighborhood of $g_0$ such that $\Lambda_0$ admits a hyperbolic continuation $\Lambda$ for all $g\in\mathcal{U}$. 

\begin{theorem}\label{t.A} If $\mathcal{U}$ is sufficiently small, then there exists a Baire residual subset $\mathcal{U}^{*}\subset \mathcal{U}$ with the following property. For every $g\in\mathcal{U}^{*}$, there exists a dense subset $\mathcal{H}_{\phi_g,\Lambda}\subset C^s(S_g N,\mathbb{R})$, $s\geq 4$, such that the function 
$$t\mapsto \textrm{dim}(L_{\phi_g, \Lambda, f}\cap (-\infty, t))$$
is continuous and 
$$\textrm{dim}(L_{\phi_g, \Lambda, f}\cap (-\infty, t)) = \textrm{dim}(M_{\phi_g, \Lambda, f}\cap (-\infty, t)) \quad \forall\,\, t\in\mathbb{R}$$
whenever $f\in \mathcal{H}_{\phi_g,\Lambda}$.
\end{theorem}

\section{Proof of the main result} 

Morally speaking, our proof of Theorem \ref{t.A} consists into a reduction to the context of the first paper of this series \cite{CMM}.

\subsection{Dimension reduction via Poincar\'e maps} The notion of \emph{good cross sections} was exploited in \cite{MoRo} to describe the dynamics of $\phi_g$ on $\Lambda$ (for $g\in\mathcal{U}$) in terms of Poincar\'e maps. More precisely, they constructed a finite number of disjoint smooth ($C^r$, $r\geq 3$) cross sections $\Sigma_i$, $1\leq i\leq k$, of the flow $\phi$ such that the $\phi$-orbit of any point of $\Lambda$ intersects $\Theta:=\bigsqcup\limits_{i=1}^k \Sigma_i$, the subset $K:=\Lambda\cap\Theta$ is disjoint from the boundary $\partial\Theta:=\bigsqcup\limits_{i=1}^k\partial\Sigma_i$, and $K$ is a horseshoe of the Poincar\'e (first return) map $\mathcal{R}:D_{\mathcal{R}}\to \Theta$ from a neighborhood $D_{\mathcal{R}}\subset\Theta$ of $K$ to $\Theta$ sending $y\in D_{\mathcal{R}}$ to the point  $\mathcal{R}(y) = \phi^{t_+(y)}(y)$ where the forward $\phi$-orbit of $y$ first hits $\Theta$. 

The relation between the Hausdorff dimensions of $K$ and $\Lambda$ is described by the following lemma (compare with Lemma 14 in \cite{MoRo}): 

\begin{lemma}\label{l.dimK} In the previous setting, one has $\textrm{dim}(\Lambda)=\textrm{dim}(K)+1$. 
\end{lemma}

\begin{proof} We cover $\Lambda$ with a finite number of tubular neighborhoods $U_l$, $1\leq l\leq m$ of compact pieces of $\phi$-orbits issued from points in $\Theta$, say $U_l=\{\phi^t(y): |t|<\gamma_l, y\in V_l\}$ where $V_l\subset \Theta-\partial\Theta$ is open and $\gamma_l\in\mathbb{R}$. 

Since $\textrm{dim}(\Lambda)=\max\limits_{1\leq l\leq m}\{\textrm{dim}(\Lambda\cap U_l)\}$ and $\textrm{dim}(K) = \max\limits_{1\leq l\leq m}\{\textrm{dim}(K\cap V_l)\}$, we can select $l_0$ and $l_1$ such that $\textrm{dim}(\Lambda)=\textrm{dim}(\Lambda\cap U_{l_0})$ and $\textrm{dim}(K)=\textrm{dim}(K\cap V_{l_1})$. Because $\Lambda\cap V_l = K\cap V_l$ and $U_l$ is a tubular neighborhood for each $1\leq l\leq m$, we also have that $\Lambda\cap U_l$ is diffeomorphic to $(K\cap V_l)\times (-\gamma_l,\gamma_l)$. 

It follows that 
$$\textrm{dim}(\Lambda) = \textrm{dim}(\Lambda\cap U_{l_0}) = \textrm{dim}(K\cap V_{l_0})+1\leq \textrm{dim}(K)+1$$ 
and 
$$\textrm{dim}(K)+1 = \textrm{dim}(K\cap V_{l_1}) +1 = \textrm{dim}(\Lambda\cap U_{l_1}) \leq \textrm{dim}(\Lambda)$$

This proves the lemma. 
\end{proof}

The dynamical Lagrange and Markov spectra of $\Lambda$ and $K$ are related in the following way. Given a function $f\in C^s(S_gN,\mathbb{R})$, $s\geq 1$, let us denote by $F=\max_{\phi} f: D_{\mathcal{R}}\to\mathbb{R}$ the function 
$$F(y):=\max\limits_{0\leq t\leq t_+(y)} f(\varphi^t(y))$$

\begin{remark} $F=\max_{\phi} f$ might not be $C^1$ in general. 
\end{remark} 

By definition: 
$$\limsup\limits_{n\to+\infty}F(\mathcal{R}^n(x)) = \limsup_{t\to+\infty} f(\phi_g^t(x))$$ 
and 
$$\sup\limits_{n\in\mathbb{Z}} F(\mathcal{R}^n(x)) = \sup\limits_{t\in\mathbb{R}} f(\phi_g^t(x))$$ 
for all $x\in K$. In particular, 
$$L_{\phi_g,\Lambda,f} = L_{\mathcal{R}, K, F} \quad \textrm{ and } \quad M_{\phi_g, \Lambda, f} = M_{\mathcal{R}, K, F}$$

This reduces Theorem \ref{t.A} to the following statement: 

\begin{theorem}\label{t.B} In the setting of Theorem \ref{t.A}, if $\mathcal{U}$ is sufficiently small, then there exists a Baire residual subset $\mathcal{U}^*\subset\mathcal{U}$ such that, for each $g\in\mathcal{U}^*$, one can find a dense subset $\mathcal{H}_{\phi_g,\Lambda}\subset C^s(S_gN,\mathbb{R})$, $s\geq 4$, so that the function 
$$t\mapsto \textrm{dim}(L_{\mathcal{R}, K, \max_{\phi_g} f}\cap (-\infty, t))$$
is continuous and 
$$\textrm{dim}(L_{\mathcal{R}, K, \max_{\phi_g} f}\cap (-\infty, t)) = \textrm{dim}(M_{\mathcal{R}, K, \max_{\phi_g} f}\cap (-\infty, t)) \quad \forall\,\, t\in\mathbb{R}$$
whenever $f\in \mathcal{H}_{\phi_g,\Lambda}$. 
\end{theorem}

The proof of Theorem \ref{t.B} starts as follows. Let $\{R_a\}_{a\in\mathcal{A}}$ be a Markov partition consisting of rectangles $R_a\simeq I_a^s\times I_a^u$ delimited by compact pieces $I_a^s$, resp. $I_a^u$, of stable, resp. unstable, manifolds of a finite collection of $\mathcal{R}$-periodic points of $K\subset\Theta$.  

Recall that the stable and unstable manifolds of $K$ can be extended to locally $\mathcal{R}$-invariant $C^{1+\varepsilon}$-foliations in $D_{\mathcal{R}}$ for some $\varepsilon>0$. These foliations induce projections $\pi_a^u:R_a\to I_a^s\times\{i_a^u\}$ and $\pi_a^s: R_a\to \{i_a^s\}\times I_a^u$ of the rectangles into the connected components $I_a^s\times\{i_a^u\}$ and $\{i_a^s\}\times I_a^u$ of the stable and unstable boundaries of $R_a$ where $i_a^u\in\partial I_a^u$ and $i_a^s\in\partial I_a^s$ are fixed arbitrarily. In this way, we obtain stable and unstable Cantor sets 
$$K^s=\bigcup\limits_{a\in\mathcal{A}}\pi_a^u(K\cap R_a) \quad \textrm{ and } \quad K^u=\bigcup\limits_{a\in\mathcal{A}}\pi_a^s(K\cap R_a)$$
associated to $K$. 

In the sequel, we will analyze the sets 
$$K_t:=\{y\in K: m_{\mathcal{R},K,\max_{\phi} f}(y)\leq t\},$$
$$K_t^s:=\bigcup\limits_{a\in\mathcal{A}}\pi_a^u(K_t\cap R_a) \quad \textrm{ and } \quad K_t^u:=\bigcup\limits_{a\in\mathcal{A}}\pi_a^s(K_t\cap R_a)$$

\subsection{Upper-semicontinuity}

Denote by $D_s(t)$ and $D_u(t)$ the limit capacities of $K_t^s$ and $K_t^u$. As it was shown in \cite[Proposition 2.6]{CMM}, an elementary compactness argument reveals that:

\begin{proposition}\label{p.Du-upper-sc} For any $g\in\mathcal{U}$ and $f\in C^0(S_gN,\mathbb{R})$, the functions $t\mapsto D_u(t)$ and $t\mapsto D_s(t)$ are upper semicontinuous. 
\end{proposition} 

Therefore, it remains study the lower semi-continuity of $D_s(t)$ and $D_u(t)$ and their relations with  $L_{\mathcal{R}, K, \max_{\phi_g} f}\cap (-\infty, t)$ and $M_{\mathcal{R}, K, \max_{\phi_g} f}\cap (-\infty, t)$. For this sake, we introduce the Baire generic sets $\mathcal{U}^*$ and $\mathcal{H}_{\phi_g,\Lambda}$ in the statement of Theorem \ref{t.B}. 

\subsection{Description of $\mathcal{U}^*$}

We say that $g\in\mathcal{U}^*$ whenever every subhorseshoe $\widetilde{K}\subset K_g$ satisfies the so-called \emph{property (H$\alpha$)} of Moreira-Yoccoz \cite{MY} and possesses a pair of periodic points whose logarithms of unstable eigenvalues are incommensurable, where  $K_g$ denoted the hyperbolic continuation of $K$.

The set $\mathcal{U}^*$ was defined so that Moreira's dimension formula \cite[Corollary 3]{Mo} implies the following result:

\begin{proposition}\label{p.dimension-formula} Suppose that $g\in\mathcal{U}^*$. Then, given any  subhorseshoe $\widetilde{K}\subset K$ and any $C^1$ function $H:D_{\mathcal{R}}\to\mathbb{R}$ whose gradient is transverse to the stable and unstable directions of $\mathcal{R}$ at some point of $\widetilde{K}$, one has 
$$\textrm{dim}(H(\widetilde{K})) = \min\{\textrm{dim}(\widetilde{K}),1\}$$
\end{proposition} 

For later use, we observe that $\mathcal{U}^*$ is a topologically large subset of $\mathcal{U}$: 

\begin{lemma}\label{l.U*generic} $\mathcal{U}^*$ is a Baire generic subset of $\mathcal{U}$. 
\end{lemma} 

\begin{proof} By the results in Subsection 4.3 and Section 9 of \cite{MY}, every subhorseshoe $\widetilde{K}\subset K$ satisfies the property (H$\alpha$) whenever the so-called \emph{Birkhoff invariant} (cf. \cite[Appendix A]{MY}) of all periodic points of $\mathcal{R}$ in $K$ are non-zero. As it turns out, the non-vanishing of Birkhoff invariant is an open, dense and conjugation invariant condition on the third jet of a germ of area-preserving automorphism of $(\mathbb{R}^2,0)$ (compare with Lemma 32 in \cite{MoRo}). It follows from Klingenberg-Takens theorem \cite[Theorem 1]{KT} that the subset $\mathcal{V}$ of $g\in\mathcal{U}$ such that every subhorseshoe $\widetilde{K}\subset K$ satisfies the property (H$\alpha$) is $C^r$-Baire generic (for all $r\geq 4$). 

On the other hand, given any pair $p$ and $q$ of distinct periodic orbits in $K$, if we denote by $\gamma_p$ and $\gamma_q$ the corresponding $g$-geodesics on $N$, then we can select a piece $l\subset\gamma_p$ disjoint from $\gamma_q$ (because distinct geodesics intersect transversely) and we can apply Klingenberg-Takens theorem \cite[Theorem 2]{KT} to (the first jet of the Poincar\'e map along) $l$ to ensure that the logarithms of the unstable eigenvalues of $p$ and $q$ are incommensurable for a $C^r$-Baire generic subset $\mathcal{W}_{p,q}$ of $\mathcal{U}$ (for all $r\geq 2$). 

It follows that the subset 
$$\mathcal{U}^{**}=\mathcal{V}\cap\bigcap\limits_{\substack{p,q\in \textrm{Per}(\mathcal{R})\cap K \\ p\neq q}}\mathcal{W}_{p,q}$$ 
is a countable intersection of $C^r$-Baire generic subsets (for all $r\geq 4$) such that $\mathcal{U}^{**}\subset\mathcal{U}^*$. This proves the lemma.  
\end{proof}

\subsection{Description of $\mathcal{H}_{\phi_g,\Lambda}$} Let $\mathcal{H}_{\phi_g, \Lambda}$ be the set of functions $f$ such that there exists a finite collection $J$ of $C^1$-curves in $\Theta$ so that, for each $n\in \mathbb{N}$, the complement $V_n$ of the $1/n$-neighborhood of $J$ in $\Theta$ contains a finite collection $L_n$ of $C^1$-curves with the property that $F=\max_{\phi} f$ is $C^1$ on $V_n\setminus L_n$ and the gradient of $F|_{V_n\setminus L_n}$ is transverse to the stable and unstable directions of $\mathcal{R}$ at all points of $K\cap (V_n\setminus L_n)$. 

We want to show that: 
\begin{lemma}\label{l.H} $\mathcal{H}_{\phi_g, \Lambda}$ is dense. 
\end{lemma}

For this sake, we need two auxiliary sets $\mathcal{M}_{\phi_g, \Lambda} \subset \mathcal{N}_{\phi_g, \Lambda}$ of functions defined as follows. 

Once again we cover $\Lambda$ with a finite number of tubular neighborhoods $U_l$, $1\leq l\leq m$ whose boundaries are the good cross-sections $\Theta=\bigsqcup\limits_{i=1}^k\Sigma_i$ mentioned above. For each $l$, let us fix coordinates $(x_1(l),x_2(l),x_3(l))$ on $U_l$ such that $x_3(l)$ is the flow direction and $U_l\cap\Theta=\{x_3(l)=0\}\cup\{x_3(l)=1\}$. 

\begin{definition}\label{d.N}
We say that $f\in\mathcal{N}_{\phi_g, \Lambda}$ whenever: 
\begin{itemize}
\item[(i)] $0$ is a regular value of the restriction of $\frac{\partial f}{\partial x_3(l)}$ to $U_l\cap\Theta$; 
\item[(ii)] $0$ is a regular value of $\frac{\partial^3 f}{\partial x_3(l)^3}$; 
\item[(iii)] $0$ is a regular value of the functions $\frac{\partial^2 f}{\partial x_3(l)^2}$ and $\frac{\partial^2 f}{\partial x_3(l)^2}|_{\{\frac{\partial^3 f}{\partial x_3(l)^3}=0\}}$; 
\item[(iv)] $0$ is a regular value of the functions $\frac{\partial f}{\partial x_3(l)}|_{\{\frac{\partial^2 f}{\partial x_3(l)^2}=0\}}$ and $\frac{\partial f}{\partial x_3(l)}|_{\{\frac{\partial^3 f}{\partial x_3(l)^3}=0\}\cap \{\frac{\partial^2 f}{\partial x_3(l)^2}=0\}}$,
\end{itemize} 
for each $1\leq l\leq m$. 
\end{definition}

\begin{lemma}\label{l.N} $\mathcal{N}_{\phi_g,\Lambda}$ is dense. 
\end{lemma}

\begin{proof} Given a function $f$, let us consider the three-parameter family 
$$f_{a,b,c}(x_1,x_2,x_3) = f(x_1,x_2,x_3) - c x_3^3/6 - b x_3^2/2 -a x_3$$
where $a,b,c\in\mathbb{R}$. 

By Sard's theorem, we can fix first a very small regular value $c\approx 0$ of $\frac{\partial^3 f}{\partial x_3^3}$, then a very small regular value $b\approx 0$ of both $\frac{\partial^2 f}{\partial x_3^2} - c x_3$ and its restriction to $\{\frac{\partial^3 f}{\partial x_3(l)^3}=c\}$, and finally a very small regular value $a\approx 0$ of $(\frac{\partial f}{\partial x_3} - c x_3^2/2 - b x_3)|_{\{\frac{\partial^2 f}{\partial x_3^2} - cx_3 = b\}}$, $(\frac{\partial f}{\partial x_3} - c x_3^2/2 - b x_3)|_{\{\frac{\partial^3 f}{\partial x_3^3} = c\}\cap\{\frac{\partial^2 f}{\partial x_3^2} - cx_3 = b\}}$ and $(\frac{\partial f}{\partial x_3} - c x_3^2/2 - b x_3)|_{\{x_3 = 0\}\cup\{ x_3 = 1\}}$. 

For a choice of parameters $(a,b,c)$ as above, we have that $f_{a,b,c}\in\mathcal{N}_{\phi_g,\Lambda}$: indeed, this happens because $\frac{\partial^3 f_{a,b,c}}{\partial x_3^3} = \frac{\partial^3 f}{\partial x_3^3} - c$, $\frac{\partial^2 f_{a,b,c}}{\partial x_3^2} = \frac{\partial^2 f}{\partial x_3^2} - c x_3$ and $\frac{\partial f_{a,b,c}}{\partial x_3} = \frac{\partial f}{\partial x_3} - c x_3^2/2 - b x_3$. Since $f_{a,b,c}$ is arbitrarily close to $f$, this proves the lemma. 
\end{proof}

By definition, if $f\in\mathcal{N}_{\phi_g, \Lambda}$, then $\mu_l:=\{\frac{\partial f}{\partial x_3(l)}=0\}\cap U_l$ is a curve (thanks to (i)), and $J_l:=\{\frac{\partial f}{\partial x_3(l)}=0\}\cap \{\frac{\partial^2 f}{\partial x_3(l)^2}=0\}$ is a curve intersecting the surface $\{\frac{\partial^3 f}{\partial x_3(l)^3}=0\}$ at a finite set $\Pi_l$ of points (thanks to (ii), (iii) and (iv)). 

Note that if $(x_1,x_2,0), (x_1,x_2,1)\notin \mu_l$ and the piece of orbit $(x_1,x_2,z)$, $0\leq z\leq 1$, doesn't intersect $J_l$, then there is a neighborhood $V$ of $(x_1, x_2, 0)\in U_l\cap\Theta$ and a finite collection of disjoint graphs $\{(x,y,\psi_j(x,y)): (x,y,0)\in V\}$, $1\leq j\leq n$ such that if $F(x_1',x_2',t')=\max_{\phi} f(x_1',x_2',t')$ with $(x_1',x_2',0)\in V$, then $t'=\psi_j(x_1',x_2')$ for some $j$. 

\begin{definition}\label{d.M} We say that $f\in\mathcal{M}_{\phi_g, \Lambda}$ if $f\in\mathcal{N}_{\phi_g, \Lambda}$ and there exists a finite collection $J$ of $C^1$-curves in $\Theta$ so that, for each $n\in \mathbb{N}$, the complement $V_n$ of the $1/n$-neighborhood of $J$ in $\Theta$ contains a finite collection $L_n$ of $C^1$-curves with the property that for each $y\in V_n\setminus L_n$, there is an unique $0\leq t(y)\leq t_+(y)$ with $F(y)=f(\phi^{t(y)}(y))$, and, moreover, the function $y\mapsto \phi^{t(y)}(y)$ is $C^1$ on $V_n\setminus L_n$. 
\end{definition}

\begin{lemma}\label{l.M} $\mathcal{M}_{\phi_g,\Lambda}$ is dense. 
\end{lemma}

\begin{proof} Consider $f\in\mathcal{N}_{\phi_g,\Lambda}$ as above. Our discussion so far says that the curves $\mu_l$ and the projections of the curves $J_l$ in the flow direction ($x_3$-coordinate) is a  finite union $J$ of $C^1$ curves contained in $\Theta$ such that, for each $y\in D_\mathcal{R}\setminus J$, the value $F(z)$ for $z$ near $y$ is described by the values of $f$ at a finite collection of graphs transverse to the flow direction. 

In other terms, using the notation in the paragraph before Definition \ref{d.M}, our task is reduced to perturb $f$ in such a way that $F(x_1',x_2',t')$ are given by the values of $f$ on an \emph{unique} graph $(x_1',x_2',\psi(x_1',x_2'))$. 

In this direction, we employ the argument from Lemma 19 in \cite{MoRo}. More precisely, given $N\in\mathbb{N}$, the value of $F$ at any point $(x,y)\in V_N$ is described by finitely many disjoint  graphs $\psi_j$, $1\leq j\leq n$ (where $n$ depends on $N$). As it is explained in Lemma 19 in \cite{MoRo}, we can perform small perturbations of $f$ on $V_N$ in such a way that $0$ is a simultaneous regular value of the functions $(x_1,x_2)\mapsto g_{ji}(x_1,x_2):= f(x_1,x_2,\psi_j(x_1,x_2))-f(x_1,x_2,\psi_i(x_1,x_2))$ for all choices of $1\leq j<i\leq n$. In this situation, $L_n=\bigcup\limits_{1\leq j<i\leq n} g_{ji}^{-1}(0)$ is a finite collection of $C^1$-curves such that, for each $y\in V_n\setminus L_n$, the values of $F$ near $y$ are described by the values of $f$ at an unique graph. Hence, for each $y\in V_n\setminus L_n$, one has that $F(y)=f(\phi^{t(y)}(y))$ for an unique $0\leq t(y)\leq t_+(y)$ depending in a  $C^1$ way on $y$. 

This shows the lemma. 
\end{proof}

At this point, we are ready to establish the denseness of $\mathcal{H}_{\phi_g, \Lambda}$: 

\begin{proof}[Proof of Lemma \ref{l.H}] Given a function $f\in C^s(S_gN,\mathbb{R})$, we apply Lemma \ref{l.M} in order to perform a preliminary perturbation so that $f\in\mathcal{M}_{\phi_g,\Lambda}$. In this context, our task is simply to prove that some appropriate perturbations of $f$ render the gradient of $F=\max_{\phi} f$ transverse to the stable and unstable directions at all points of $K\setminus(\bigcup\limits_{n\in\mathbb{N}} L_n\cup J)$. 

For this sake, we fix $n\in\mathbb{N}$ and consider a point $x\in K\cap(V_n\setminus L_n)$. Recall that, in a small neighborhood of $x$, the values of $F=\max_{\phi} f$ are given by the values of $f$ on a graph $(x_1,x_2,\psi(x_1,x_2))$. Since the Hausdorff dimension of $K$ is strictly smaller than one (cf. Lemma \ref{l.dimK}), we can employ the argument in Proposition 2.7 in \cite{CMM} to find arbitrarily small vectors $v=(v_1, v_2)\in\mathbb{R}^2$ such that the functions $f_v(x_1,x_2,t):=f(x_1,x_2,t)-v_1x_1-v_2x_2$ near the graph $(x_1,x_2,\psi(x_1,x_2))$ (and coinciding with $f$ elsewhere) have the property that the gradient of $F_v:=\max_{\phi} f_v$ is transverse to the stable and unstable directions of any point of $K$ close to $(x_1,x_2)$. Because $n\in \mathbb{N}$ and $x\in K\cap (V_n\setminus L_n)$ were arbitrary, the proof of the lemma is complete. 
\end{proof}

\subsection{Lower semicontinuity} The first step towards the lower semicontinuity $D_u(t)$ and $D_s(t)$ is the following analog of Proposition 2.10 in \cite{CMM}: 

\begin{proposition}\label{p.210'} Suppose that $g\in\mathcal{U}$ and $f\in\mathcal{H}_{\phi_g,\Lambda}$. Given $t\in\mathbb{R}$ such that $D_u(t)>0$, resp. $D_s(t)>0$, and $0<\eta<1$, there exist $\delta>0$ and a (complete) subhorseshoe $K'\subset K_{t-\delta}$ such that 
$$\textrm{dim}((K')^u)>(1-\eta)D_u(t) \quad \textrm{and} \quad \textrm{dim}((K')^s)>(1-\eta)D_u(t)$$
resp.
$$\textrm{dim}((K')^u)>(1-\eta)D_s(t) \quad \textrm{and} \quad \textrm{dim}((K')^s)>(1-\eta)D_s(t)$$

In particular, $D_u(t)=D_s(t)=d_u(t)=d_s(t)$ for all $t\in\mathbb{R}$. 
\end{proposition}

\begin{proof} By symmetry (i.e., replacing the flow by its inverse), it suffices to prove the statement when $D_u(t)>0$. 

We consider the construction of $K$ in terms of its Markov partition $R_a$, $a\in\mathcal{A}$, introduced above. Given an admissible\footnote{i.e., there is a point $x\in K$ such that $\mathcal{R}^i(x)\in R_{a_i}$ for all $i=0,\dots, k$. } word $\alpha=(a_0,\dots,a_k)$ on the alphabet $\mathcal{A}$, denote by $I^u(\alpha) = \pi_{a_0}(\{x\in R_{a_0}: \mathcal{R}^i(x)\in R_{a_i} \, \forall i=1,\dots, k\})$. In this setting, the \emph{unstable scale} $r^u(\alpha)$ is $\lfloor\log (1/(\textrm{length of } I^u(\alpha)))\rfloor$. 

For each $r\in\mathbb{N}$, define 
$$P^u_r := \{\alpha = (a_0,\dots,a_k) \textrm{ admissible word}:  r^u(\alpha)\geq r \textrm{ and } r^u(a_0,\dots, a_{k-1})<r\},$$
$$C^u(t,r) :=  \{\alpha\in P^u_r: I^u(\alpha)\cap K^u_t\neq\emptyset\}$$
and $N_u(t,r) := \# C^u(t,r)$. 

Of course, we have similar notions of $I^s(\beta)$, etc.  

Denote by $\tau=\eta/100$. By the definition of limit capacity, we can fix $r_0$ sufficiently large such that 
$$\left|\frac{\log N_u(t,r)}{r} - D_u(t)\right|<\frac{\tau}{6} D_u(t)$$ 
for all $r\geq r_0$. 

Recall that the fact that $f\in\mathcal{H}_{\phi_g,\Lambda}$ is associated to a finite collection $J$ of $C^1$-curves in $\Theta$ so that, for each $n\in \mathbb{N}$, the complement $V_n$ of the $1/n$-neighborhood of $J$ in $\Theta$ contains a finite collection $L_n$ of $C^1$-curves with the property that $F=\max_{\phi} f$ is $C^1$ on $V_n\setminus L_n$ and the gradient of $F|_{V_n\setminus L_n}$ is transverse to the stable and unstable directions of $\mathcal{R}$ at all points of $K\cap (V_n\setminus L_n)$. 

As it is explained in Lemma 18 in \cite{MoRo}, it is possible to select a subset $B^u(r_0)\subset C^u(t,r_0)$ such that 
$$\frac{\log\#B^u(r_0)}{r_0}\geq \frac{\log N_u(t,r_0)}{r_0} - \frac{\tau}{6} D_u(t)$$
and the subhorseshoe $K^{(r_0)}\subset K$ associated to the admissible words in $B^u$ is disjoint from $J$. 

By selecting $n_0\in\mathbb{N}$ large so that $K^{(r_0)}\subset V_{n_0}$ and by applying again the arguments in Lemma 18 in \cite{MoRo}, we can find a subset $B^u\subset B^u(r_0)$ such that 
$$\frac{\log\#B^u}{r_0}\geq \frac{\log B^u(r_0)}{r_0} - \frac{\tau}{6} D_u(t)$$
and the subhorseshoe $K''\subset K$ associated to the admissible words in $B^u$ is contained in $V_n\setminus L_n$. 

In summary, we obtained a subset $B^u\subset C^u(t,r_0)$ with 
$$\left|\frac{\log\#B_u}{r} - D_u(t)\right|<\frac{\tau}{2} D_u(t)$$ 
such that the subhorseshoe $K''\subset K$ associated to $B^u$ is contained in $V_{n_0}\setminus L_{n_0}$ and, \emph{a fortiori}, the gradient of $F=\max_{\phi} f$ is transverse to the stable and unstable directions at all points of $K''$. 

In this scenario, we can run the same arguments from Proposition 2.10 in \cite{CMM} in order to locate a subhorseshoe $K'\subset K''$ with the desired features. 
\end{proof}

At this stage, we are ready to show the lower semicontinuity of $D_u(t)$ and $D_s(t)$. 

\begin{proposition}\label{p.Du-lower-sc} For $g\in\mathcal{U}^*$ and $f\in\mathcal{H}_{\phi_g,\Lambda}$, the functions $t\mapsto D_u(t)$ and $t\mapsto D_s(t)$ are lower semicontinuous and 
$$D_s(t)+D_u(t) = 2 D_u(t) = \textrm{dim}(L_{\mathcal{R}, K, \max_{\phi_g} f}\cap (-\infty, t)) = \textrm{dim}(M_{\mathcal{R}, K, \max_{\phi_g} f}\cap (-\infty, t))$$
\end{proposition}

\begin{proof} Consider $t\in\mathbb{R}$ with $D_u(t)>0$ and fix $\eta>0$. By Proposition \ref{p.210'}, we can find $\delta>0$ and a subhorseshoe $K'\subset K_{t-\delta}$ such that 
$$(1-\eta)(D_u(t)+D_s(t)) = 2(1-\eta)D_u(t) \leq \textrm{dim}(K')$$

Since the gradient of $F=\max_{\phi} f$ is transverse to the stable and unstable directions of $K'$ (cf. the proof of Proposition \ref{p.210'} above), we can use Proposition 2.16 in \cite{CMM} to get that, for each $\varepsilon>0$, there exists a subhorseshoe $K'_{\varepsilon}\subset K'$ with $\textrm{dim}(K'_{\varepsilon})\geq \textrm{dim}(K')-\varepsilon$ and a $C^1$ height function $H_{\varepsilon}$ whose gradient is transverse to the stable and unstable directions of $K'_{\varepsilon}$ such that 
$$H_{\varepsilon}(K'_{\varepsilon})\subset\ell_{\mathcal{R},\max_{\phi} f}(K')$$ 

By Proposition \ref{p.dimension-formula}, it follows that 
$$\textrm{dim}(K')-\varepsilon\leq \textrm{dim}(K'_{\varepsilon}) = \textrm{dim}(H_{\varepsilon}(K'_{\varepsilon}))\leq \textrm{dim}(\ell_{\mathcal{R},\max_{\phi} f}(K'))$$ 
for all $\varepsilon>0$. In particular, $\textrm{dim}(K')\leq \textrm{dim}(\ell_{\mathcal{R},\max_{\phi} f}(K'))$. 

Because $K'\subset K_{t-\delta}$, one has $\ell_{\mathcal{R},\max_{\phi} f}(K')\subset L_{\phi_g, \Lambda, f}\cap (-\infty, t-\delta)$. Thus, our discussion so far can be summarized by the following estimates: 
\begin{eqnarray*}
2(1-\eta)D_u(t) &\leq& \textrm{dim}(K') \leq \textrm{dim}(\ell_{\mathcal{R},\max_{\phi} f}(K')) \\ 
&\leq& \textrm{dim}(L_{\mathcal{R}, K, \max_{\phi_g} f}\cap (-\infty, t-\delta)) \leq \textrm{dim}(M_{\mathcal{R}, K, \max_{\phi_g} f}\cap (-\infty, t-\delta)) \\ 
&\leq & \textrm{dim}(\text{max}_{\phi_{g}}f(K_{t-\delta})\leq 2 D_u(t-\delta)
\end{eqnarray*} 

This proves the proposition. 
\end{proof}

\subsection{End of proof of Theorem \ref{t.B}} Let $g\in\mathcal{U}^*$ and $f\in\mathcal{H}_{\phi_g,\Lambda}$. Note that $\mathcal{U}^*$ is a residual subset of $\mathcal{U}$ thanks to Lemma \ref{l.U*generic} and $\mathcal{H}_{\phi_g,\Lambda}$ is dense in $C^s(S_gN,\mathbb{R})$ for $s\geq 4$ thanks to Lemma \ref{l.H}. 

By Propositions \ref{p.Du-upper-sc} and \ref{p.Du-lower-sc}, the function 
$$t\mapsto D_s(t) = D_u(t) = \frac{1}{2}\textrm{dim}(L_{\mathcal{R}, K, \max_{\phi_g} f}\cap (-\infty, t)) = \frac{1}{2}\textrm{dim}(M_{\mathcal{R}, K, \max_{\phi_g} f}\cap (-\infty, t))$$ 
is continuous. 

This completes the proof of Theorem \ref{t.B} (and, \emph{a fortiori}, Theorem \ref{t.A}). 

\section*{Acknowledgement}

We would like to thank Carlos Matheus for very helpful discussions and suggestions during the preparation of this paper.

\end{document}